\documentclass[12pt]{amsart}
\usepackage{amsmath}
\usepackage{amssymb}
\usepackage{ifthen}
\usepackage{datetime}
\usepackage[english]{babel}
\usepackage{epsfig}
\nonstopmode

\usepackage{mathrsfs}
\usepackage{graphicx}
\usepackage{latexsym}
\usepackage{verbatim}

\numberwithin{equation}{section}

\newcommand {\mo} {\,{\rm mod}\,}
\newcommand{\diam}{{\rm diam}\,}

\newcommand{\loc}{{\rm loc}}
\newcommand{\adm}{{\rm adm}}
\newcommand{\SO}{{\rm SO}}

\newcommand {\Sn} {{\overline{\mathbb R}^n}}

\newcommand {\ring} {{\mathcal R}}

\newcommand {\R} {\mathbb R}
\newcommand {\B} {\mathbb B}
\newcommand {\Bno} {{\mathbb B}_0^n}
\newcommand {\pcub} {\overline{\mathbb B}_0^n}

\newcommand{\M}{\mathsf{M}}

\newcommand {\A} {\mathcal A}

\newcommand {\aand} {\quad\text{and}\quad}

\newcommand {\Sph} {{\mathbb{S}^{n-1}}}

\theoremstyle{theorem}

\numberwithin{equation}{section}

\newenvironment{thm}[1][]{%
\refstepcounter{equation}%
\medskip%
\noindent%
\arabic{section}.\arabic{equation} {\bf Theorem}%
\ifthenelse{\equal{#1}{}}{}{ (#1)}%
{\bf .}
\itshape}{\medskip}

\newenvironment{cor}[1][]{%
\refstepcounter{equation}%
\medskip%
\noindent%
\arabic{section}.\arabic{equation} {\bf Corollary}%
\ifthenelse{\equal{#1}{}}{}{ (#1)}%
{\bf .}
\itshape}{\medskip}

\newenvironment{lem}[1][]{%
\refstepcounter{equation}%
\medskip%
\noindent%
\arabic{section}.\arabic{equation} {\bf Lemma}%
\ifthenelse{\equal{#1}{}}{}{ (#1)}%
{\bf .}
\itshape}{\medskip}

\newenvironment{rem}[1][]{%
\refstepcounter{equation}%
\medskip%
\noindent%
\arabic{section}.\arabic{equation} {\bf Remark}%
\ifthenelse{\equal{#1}{}}{}{ (#1)}%
{\bf .}
}{\medskip}

\renewcommand{\subsection}[1]%
{\smallskip \noindent \refstepcounter{equation}\arabic{section}.\arabic{equation} {\bf #1.}}

\begin{document}
\title[Cavitation]{Modulus estimates and cavitation in higher dimensions}

\author[A. Golberg]{Anatoly Golberg}
\address{School of Mathematical Sciences \\
Holon Institute of Technology \\
52 Golomb St., P.O.B. 305, Holon 5810201, Israel \\
ORCID 0009-0002-8785-7463
}
\email{golberga@hit.ac.il}

\author[V. Gutlyanski\u\i]{Vladimir Gutlyanski\u\i}
\address{NAS of Ukraine \\
Institute of Applied Mathematics and Mechanics \\
1 Dobrovolskogo St., Slavyansk 84100, Ukraine \\
ORCID 0000-0002-8691-4617
}
\email{vgutlyanskii@gmail.com}

\author[V. Ryazanov]{Vladimir Ryazanov}
\address{NAS of Ukraine \\
Institute of Applied Mathematics and Mechanics \\
1 Dobrovolskogo St., Slavyansk 84100, Ukraine \\
ORCID 0000-0002-4503-4939
}
\email{vl.ryazanov1@gmail.com}

\author[T. Sugawa]{Toshiyuki Sugawa}
\address{Graduate School of Information Sciences \\
Tohoku University \\
Aoba-ku, Sendai 980-8579, Japan \\
ORCID 0000-0002-3429-5498
}
\email{sugawa@math.is.tohoku.ac.jp}

\begin{abstract}
We explore the phenomenon of cavitation in higher-dimensional elasticity, defining it as the mapping of a punctured ball onto a non-degenerate ring domain. Crucially, for the class of locally quasiconformal mappings (or more general mappings) defined on the punctured ball $0<|x|<1$ in $\mathbb R^n$ that we examine, cavitation is equivalent to the failure of continuous extension to the origin. While existing modulus estimates prove insufficient for reliably detecting cavitation in this setting, our study establishes refined modulus bounds. This is achieved by introducing a novel directional dilatation which, in conjunction with the known angular dilatation, overcomes the limitations of previous methods. We illustrate our theoretical findings with several examples that demonstrate both cavitation occurrence and its absence.
\end{abstract}
\maketitle


\medskip

\bigskip

\section{Introduction}\label{sec1}

\medskip

Geometric Function Theory, particularly, Mapping Theory, shares profound connections with Elasticity Theory.
When describing the transformation of a solid body $G\subset \mathbb R^n$ by a mapping $f:G\to \mathbb R^n,$
the continuity of $f$ can be interpreted as the absence of breaks or cavities within the transformed body $f(G).$
Conversely, identifying necessary and sufficient conditions for ``cavitation''
— the creation of empty spaces within a solid object — is a topic of significant interest.

Let $f$ be an injective continuous map of the punctured open unit ball
$\mathbb B^n_0:=\mathbb B^n\setminus\{0\}\subset\R^n$ into $\R^n.$
Note that, by Brouwer's invariance of domain theorem, $V=f(\Bno)$ is a domain, and $f$ is a homeomorphism of $\Bno$ onto $V.$
We will assume that $f$ maps the inside of the sphere $S_r=\{|x|=r\}$ into the inside of the
closed surface $f(S_r).$
In this paper, we will say that $f$ admits \textit{cavitation} at the origin if $V$ is a non-degenerate ring,
in other words, the bounded connected component of $\R^n\setminus V$ is a non-degenerate continuum.
Our focus will be on locally quasiconformal mappings (see \cite{GMR23} for the planar case).
From a mapping theory perspective, cavitation can be viewed as a failure of boundary extension to the origin,
a notoriously challenging problem in the field.

In higher dimensions, the behavior of a mapping is mainly described in terms of two key quantities -
the outer and inner dilatation coefficients - defined respectively as
\begin{equation*}
K_f(x)\,=\,\frac{\|f'(x)\|^n}{J_f(x)}\,\qquad\text{and}\qquad L_f(x)\,=\,\frac{J_f(x)}{l_f(x)^n}\,.
\end{equation*}
Here,
\begin{equation*}
l_f(x)=\min_{h\in\Sph}|f'(x)h| \qquad\text{and}\qquad
\|f'(x)\|=\max_{h\in\Sph}|f'(x)h|
\end{equation*}
represent the minimal stretching and the maximal stretching, respectively.
Here and hereafter, $\Sph$ denotes the unit sphere of $\R^n.$
We recall that a homeomorphism $f$ of a domain $G$ in $\R^n$ onto another $G'\subset\R^n$ is called
$K$-quasiconformal if $f\in {\rm ACL}(G)$ and $K_f(x)\le K$ for a.e.\,$x\in G.$
See \cite{Vai71} or \cite{GMP17} for details.
A homeomorphism $f:G\to G'$ is simply called quasiconformal if it is $K$-quasiconformal for some $K<+\infty.$
It is well known that quasiconformal homeomorphisms do not admit cavitation
(see, for instance, \cite[Thm 17.3]{Vai71}).
In order to observe cavitation phenomena,
we will consider locally quasiconformal embeddings of the punctured closed unit ball
$\pcub$ into $\R^n.$
Here and hereafter, we will say that $f$ is a \emph{locally quasiconformal embedding}
of $\pcub$ into $\R^n$ if $f$ extends to a domain $G \,(\not\ni 0)$ containing $\pcub$ in such a way that
$f$ is quasiconformal on $\{x\in G: |x|>\varepsilon\}$ for each $\varepsilon\in(0,1).$
Note that $f\in W^{1,n}_\loc(G)$ in this case.
Throughout the paper, we  will assume that $f$ preserves the insides of rings.
More precisely, letting $S_r=\{x: |x|=r\},$ for $0<r_1<r_2<1,$ the surface $f(S_{r_1})$ is
contained in the inside of the surface $f(S_{r_2}).$

To illustrate the concept of cavitation and its absence, consider the following
two locally quasiconformal embeddings of $\pcub.$

The radial stretching
\begin{equation*}
f_1(x)\,=\,\frac{1+|x|^\alpha}{2|x|}x\,,\qquad 0<\alpha<1\,,
\end{equation*}
maps $\mathbb B^n_0$ onto the spherical annulus $\mathcal A(\frac{1}{2},1)$ clearly indicating cavitation at the origin.

Conversely, another stretching of $\mathbb B^n_0$
\begin{equation*}
f_2(x)\,=\,xe^{1-1/|x|}
\end{equation*}
does not exhibit cavitation at the origin and can be continuously extended to 0 by setting $f_2(0)=0.$

Our investigation into cavitation relies on bounds for the modulus of families of curves (or, in an equivalent form, in terms of the modulus of ring domains). This approach has proven highly successful in the planar case, tracing back to the seminal work of Gr\"otzsch \cite{Gro28}. While such estimates exist in higher dimensions for quasiconformal mappings (an upper bound) \cite{BGMV03} and mappings of finite area distortion (both lower and upper ones) \cite{KR08, MRSY09}, we demonstrate in Section~\ref{sec4} that these existing bounds are often insufficient to definitively identify cavitation or non-cavitation.

To address this limitation, we introduce finer bounds that incorporate two directional dilatations. In the plane, various directional dilatations have been indispensable tools for analyzing mapping features (e.g., \cite{AC71_1, BJ94, Gol18,  GMSV05, GS01, RW65, RSY05}). To the best of our knowledge, the first multidimensional directional dilatation was introduced in \cite{AC71_2}. Our prior work \cite{GG09} extended the angular dilatation $D_\mu(z,z_0)$ from $\mathbb C$ to $\mathbb R^n;$ cf. \cite{GG09TM}. In \cite{Gol10, GSV25}, we also utilized the normal directional dilatation in modulus bounds.

In this paper, we introduce a novel directional dilatation, $Q_f(x,x_0),$
which in the planar case reduces to the radial dilatation $D_{-\mu}(z,z_0).$
This new quantity plays a crucial role in deriving lower estimates for the modulus of families of curves. Furthermore, our use of the modulus of families of curves allows for a relaxation of absolute continuity assumptions compared to some previous works (cf. \cite{GSV25}). The main results of this paper establish several necessary and sufficient conditions for the occurrence of cavitation (and non-cavitation), which are illustrated by the mappings $f_1$ and $f_2$ and other relevant examples.
We also show that this approach allows us to immediately derive some previously known bounds, whose original proofs relied on intricate and non-trivial techniques.


\section{Main notions and modulus estimates}\label{sec2}

\subsection{Modulus of a family of curves}
Let $\Gamma$ be a family of curves in $\mathbb R^n.$
A Borel measurable function $\varrho$ on $\mathbb R^n$ is called {\it admissible for $\Gamma$}
if $0\le \varrho<+\infty$ a.e. on $\R^n$ and if
$\int_\gamma\varrho(x)|dx|\ge1$ whenever $\gamma\in\Gamma$ is locally rectifiable.
We denote by $\adm(\Gamma)$ the set of admissible functions $\varrho$ for $\Gamma.$

The (conformal) modulus of $\Gamma$ is defined to be
\begin{equation*}
\M(\Gamma)\,=\,\inf_{\varrho\in\adm(\Gamma)} \int\limits_{\R^n}\varrho(x)^n\,dm_n(x)\,,
\end{equation*}
where $m_n$ stands for the Lebesgue measure on $\mathbb R^n.$
If a property holds for all $\gamma\in\Gamma\setminus \Gamma_0$ for a subfamily $\Gamma_0$
of $\Gamma$ with $\M(\Gamma_0)=0,$ we will say that the property holds for almost every $\gamma\in\Gamma.$


\subsection{Rings}
Throughout this paper, a continuum will mean a connected, compact and non-empty set.
A continuum is said to be non-degenerate if it contains more than one point.
A continuum $C\subsetneq\Sn$ is called {\it filled} if $\Sn\setminus C$ is connected.
For a pair of disjoint filled continua $C_0$ and $C_1$ in $\Sn,$
the set $\ring=\Sn\setminus(C_0\cup C_1)$ is open and connected and
will be called a ring domain or, simply, a {\it ring} and sometimes denoted by $\ring(C_0, C_1).$
The ring $\ring$ is said to have non-degenerate boundary if each component $C_j$
is a non-degenerate continuum.
We will say that a ring $\ring=\ring(C_0,C_1)$ separates a set $E$
if $\ring\cap E=\varnothing$ and if
$C_j\cap E\ne\varnothing$ for $j=0,1.$
For non-empty sets $E_0, E_1,$
$\ring$ is said to separate $E_0$ from $E_1$ if $E_j\subset C_j$ for $j=0,1.$
In the sequel, when $\ring\subset\R^n,$
we will assume conventionally that $\infty\in C_1$ unless otherwise stated.

Let $\Gamma_\mathcal R$ be the family of all curves joining $C_0$ and $C_1$ in $\mathcal R.$
Then the modulus (also called the module) of  $\mathcal R$ is defined by
\begin{equation*}
\mo \mathcal R\,=\,\left[\frac{\omega_{n-1}}{\M(\Gamma_\mathcal R)}\right]^{1/(n-1)}
\,,
\end{equation*}
where $\omega_{n-1}$ denotes the area of the $(n-1)$-dimensional unit sphere \cite[p.~IX]{Vai71}.

For the spherical ring $\A(r,R)=\{x\in\mathbb R^n\,:\,r<|x-a|<R\}$ centered at a point $a,$
we have $\mo \A(r,R)=\log(R/r)$ (see, e.g. \cite[pp. 22--23]{Vai71}).

A ring $\ring'$ is said to be a {\it subring} of a ring $\ring$ if $\ring'\subset\ring$ and if
$\ring'$ separates $\Sn\setminus\ring.$
By the monotonicity of the moduli of curve families, we have the inequality $\mo \ring'\le \mo\ring$
in this case (the Comparison Principle).

In the planar case, thanks to the uniformization theorem, most domains admit
the hyperbolic metric, which is conformally invariant and offers a very powerful tool
in Geometric Function Theory.
However, in higher dimensions, the lack of the hyperbolic metric
necessitates the use of spherical subrings, whose moduli can be explicitly calculated,
unlike those of arbitrary subrings. The sharp result for the existence of spherical
subrings (Teichm\"uller's theorem on separating rings) has been established in \cite{GSV20}
and recently applied in \cite{GSV25} to boundary correspondence problems.
For more on the modulus technique, see also \cite{MRSY09}
and \cite{Sug10}.


\subsection{Multidimensional Teichm\"uller Theorem}
Following \cite{GSV20}, we recall the quantity $A_n$ defined by
\begin{equation}\label{eq:An}
A_n\,=\,\sup_{1<t<+\infty}\left[\mo \ring_{T,n}(t)-\log t\right]\,,
\end{equation}
where $\ring_{T,n}(t)=\mathcal R([-e_1,0], [te_1,\infty]),$ $t>0,$ stands for the Teichm\"uller ring.
It is known that $A_2=\pi$ and $A_n\le 2\log(1+\sqrt2)+\frac2{n-1}\log 2+\frac{2n(n-2)}{n-1}$
for $n\ge 2.$

The following theorem \cite{GSV20, GSV25} can be regarded as a Teichm\"uller-type theorem for higher dimensions.

\begin{thm}\label{thm:Teich} 
Let $n\ge 2.$
Every ring domain $\mathcal R$ separating a given point $x_0$ in $\mathbb R^n$ from $\infty$
with $\mo \mathcal R>A_n$ contains a spherical ring $\mathcal A$ centered at $x_0$
such that $\mo \A\ge \mo \ring - A_n.$
The constant $A_n$ is sharp.
\end{thm}


\subsection{Directional dilatations}\label{Direc}
Using the directional derivative of $f$ in a direction $h, h\ne 0$, at $x$,
given by
\begin{equation*}
\partial_h f(x)\,=\,\lim\limits_{t\to 0^+}\frac{f(x+th)-f(x)}{t}\,,
\end{equation*}
whenever the limit exists, we define the angular dilatation as follows.
Obviously, $\partial_h f(x)=f'(x)h,$ if $f$ is differentiable at $x,$
where $f'(x)$ denotes the Jacobi matrix of $f$ at $x.$
Throughout this paper, a point $x=(x_1,\dots,x_n)\in\R^n$ is often
treated as a $n\times 1$ matrix; namely, a column vector, in matrix computations.

Let $f$ be an orientation-preserving homeomorphism of a domain $G$
in $\R^n$ onto another domain $G'$ in $\R^n$
and let $x\in G$ be a regular point of $f.$
For a point $x_0\in\mathbb R^n$, we define the \textit{angular dilatation}
of the mapping $f$ at $x\in G,$ $x\ne x_0,$ \textit{with respect to} $x_0,$
introduced in \cite{GG09}, by
\begin{equation*}
D_f(x,x_0)\,=\,\frac{J_f(x)}{\ell_f(x,x_0)^n}\,,
\qquad\ell_f(x,x_0)\,=\,\min\limits_{h\in\Sph}\frac{|\partial_h f(x)|}{|h\cdot u|}\,,
\end{equation*}
where $u=(x-x_0)/|x-x_0|$ stands for the unit radial vector.
The dual dilatation
\begin{equation*}
T_f(x,x_0)=\biggl[\frac{\mathcal L_f(x,x_0)^n}{J_f(x)}\biggr]^{1/(n-1)}, \qquad
\mathcal L_f(x,x_0)=\max\limits_{h\in\Sph}\bigl(|\partial_h f(x)||h \cdot u|\bigr),
\end{equation*}
was introduced in \cite{Gol10} (see also \cite{GSV25}).

Note that for the planar case the angular dilatation is represented by
\begin{equation*}
D_f(z,z_0)\,=\,D_{\mu}(z,z_0)\,=\,
\frac{\left|1-\mu(z)\frac{z-z_0}{\bar z- \bar z_0}\right|^2}{1-|\mu(z)|^2}\,,
\end{equation*}
where $\mu$ is the complex dilatation $f_{\bar z}/f_z$ of $f$ (see \cite[Lemma 2.10]{RSY05}).

The quantity $T_f(x,x_0)$ was used to bound the modulus of a ring from above
(see \cite[Lemma 3]{GSV25}), but the estimate relies on the modulus of
the family of hypersurfaces separating the two boundary components of the ring.
To this end, we needed to assume several additional properties of the mapping.
Moreover, the quantity $T_f(x,x_0)$ is not easy to compute.
In this paper, we propose a new directional dilatation as a multidimensional
counterpart of the planar radial dilatation $D_{-\mu}(z,z_0)$ by
\begin{equation*}
Q_f(x,x_0)\,=\,\left(\frac{|\partial_u f(x)|^n}{J_f(x)}\right)^{1/(n-1)}\,,
\quad
u=\frac{x-x_0}{|x-x_0|}.
\end{equation*}
The quantity $Q_f(x,x_0)$ will be called the {\it normal dilatation} of $f$
at $x\in G$ with respect to $x_0\in\R^n.$
Since $l_f(x)\le\ell_f(x,x_0)\le|\partial_u f(x)|\le \mathcal L_f(x,x_0)\le \|f'(x)\|,$
at a regular point $x$ of $f,$ the chain of inequalities
\begin{equation}\label{eq:Qf}
\begin{split}
\frac1{K_f(x)}&\le \frac1{L_f(x)^{1/(n-1)}}\le\frac1{D_f(x,x_0)^{1/(n-1)}} \\
&\le Q_f(x,x_0)\le T_f(x,x_0) \le K_f(x)^{1/(n-1)}\le L_f(x)
\end{split}
\end{equation}
holds.
In particular, $\log D_f(x,x_0)$ and $\log Q_f(x,x_0)$ are locally (essentially) bounded
on the domain of definition of a locally quasiconformal map $f.$

Note that both the angular and normal dilatations may range from
0 to $\infty,$ unlike classical dilatations which are always
greater than or equal to 1.
Clearly, $L_f(x)=K_f(x)\equiv 1$ for conformal mappings $f,$ and therefore,
both directional dilatations are also equal to 1, but not vice versa.

Consider the ``quick rotation'' of $\mathbb B^n_0$ defined by
\begin{equation*}
f_3(x)=\left(ze^{2i\log|z|},x_3,\ldots,x_n\right),
\qquad
z=x_1+ix_2,\quad 0<|x|<1,
\end{equation*}
where the first complex component is defined to be zero when $z=0$.
For almost every $x\in\mathbb B^n_0$, this mapping satisfies
$D_{f_3}(x,0)=1,$
$K_{f_3}(x)=L_{f_3}(x)=(1+\sqrt{2})^n>1;$
see \cite{GG09}.
By a direct calculation, one has $Q_{f_3}(x,0)=\left(1+4|z|^2/|x|^2\right)^{n/2(n-1)}.$

Observe also that these directional dilatations provide a reasonable kind of flexibility in various estimates,
although their concrete evaluations are much more complicated than those of the classical ones; see, e.g., \cite{Gol18}.
In what follows, we will discuss the cavitation problems only at the origin.
Therefore, we will write for brevity $D_f(x)=D_f(x,0)$ and $Q_f(x)=Q_f(x,0).$
Note that both $D_f(x)$ and $Q_f(x)$ are Borel measurable on the set
of regular points of $f$ in $G.$
Also, we note that these quantities are rotationally invariant.
We will formulate it in an explicit way.
Let $A\in\SO(n);$ that is, $A$ is an orientation-preserving linear isometry with respect to the Euclidean metric,
and let $\tilde f=A\circ f\circ A^{-1}.$
Then obviously $D_{\tilde f}(Ax)=D_f(x), T_{\tilde f}(Ax)=T_f(x)$ and $Q_{\tilde f}(Ax)=Q_f(x)$ at
a regular point $x$ of $f.$


\subsection{Radial stretching}
Let $\Phi(t)$ be a positive, strictly increasing and continuous function on $r<t<R$ for some
$0\le r<R\le +\infty.$
Then the mapping
\begin{equation*}
f(x)\,=\,\frac{\Phi(|x|)}{|x|}\cdot x\,=\,\Phi(t) u\,, ~ \text{where}~  t\,=\,|x| ~\text{and}~ u\,=\,x/|x|\,,
\end{equation*}
is a homeomorphism of the domain $\A(r,R)=\{x\in\R^n: r<|x|<R\}$ onto the annulus
$\A(\hat r,\hat R),$ where $\hat r=\lim_{t\to r}\Phi(t)$ and $\hat R=\lim_{t\to R}\Phi(t).$
Such a mapping $f$ is called a radial stretching on $r<|x|<R.$
Let us compute $D_f(x), Q_f(x)$ and $T_f(x)=T_f(x,0).$

\begin{lem}\label{lem:rs}
Assume that $\Phi(t)$ is absolutely continuous, and let $f$ be
the radial stretching defined by $f(x)=\Phi(|x|)x/|x|=\Phi(t)u.$
Then
\begin{equation*}
D_f(x)\,=\,\left(\frac{t\Phi'(t)}{\Phi(t)}\right)^{1-n}\,,\qquad
Q_f(x)\,=\,\frac{t\Phi'(t)}{\Phi(t)}=:\mathcal Q\,,
\end{equation*}
and
\begin{equation*}
T_f(x)=
\begin{cases}
\mathcal Q, & ~\text{if}~ \mathcal Q\ge 1/\sqrt 2, \\
\big(2\sqrt{1-\mathcal Q^2}\big)^{n/(1-n)}\mathcal Q^{1/(1-n)}, & ~\text{if}~ \mathcal Q<1/\sqrt 2,
\end{cases}
\end{equation*}
for $r<t=|x|<R.$
\end{lem}

\begin{proof}
By the rotational invariance of the quantities, it is enough to show it in the case
when $x=te_1=(t,0,\dots,0).$
Then, for any $t\in(r,R)$ at which $\Phi(t)$ has a positive derivative,
a straightforward computation yields
\begin{equation*}
f'(te_1)\,=\,\begin{pmatrix}
\Phi'(t) & 0 & \cdots & 0 \\
0 & \phi(t) & \cdots & 0 \\
\vdots &\vdots & \ddots &\vdots \\
0 & 0 & \cdots & \phi(t)
\end{pmatrix},
\end{equation*}
where $\phi(t)=\Phi(t)/t.$
In particular, we have $J_f(te_1)=\Phi'(t)\phi(t)^{n-1}.$
For $h=(h_1,\dots,h_n)\in\Sph,$ we have
\begin{equation*}
\frac{|\partial_h f(x)|^2}{|h\cdot e_1|^2}\,=\,
\Phi'(t)^2+\phi(t)^2\cdot\frac{h_2^2+\cdots+h_n^2}{h_1^2}\,\ge\, \Phi'(t)^2\,,
\end{equation*}
where equality holds when $h=e_1.$
Therefore, $\ell_f(te_1)=\Phi'(t)=|\partial_{e_1}f(te_1)|$
and $D_f(te_1)=[\phi(t)/\Phi'(t)]^{n-1}=[\Phi(t)/t\Phi'(t)]^{n-1}$ and
$Q_f(te_1)=D_f(te_1)^{1/(1-n)}=t\Phi'(t)/\Phi(t).$
On the other hand,
\begin{align*}
\big(|\partial_h f(x)||h \cdot e_1|\big)^2&=\,\big(\Phi'(t)^2h_1^2+\phi(t)^2h_2^2+\cdots
+\phi(t)^2h_n^2\big)h_1^2 \\
&=\,[\Phi'(t)^2-\phi(t)^2]\tau^2+\phi(t)^2 \tau=F(\tau)\,,
\end{align*}
where $\tau=h_1^2\in[0,1].$
It is elementary to check that $F(\tau)$ is maximized at $\tau=1$
if $2\Phi'(t)^2\ge \phi(t)^2;$ otherwise $F(\tau)$ is maximized at
$\tau=\tau_0:=\phi(t)^2/2(\phi(t)^2-\Phi'(t)^2)=1/2(1-\mathcal Q^2).$
In the former case, $\mathcal L_f(te_1)=|\partial_{e_1}f(te_1)|=\Phi'(t).$
Hence $T_f(te_1)=Q_f(te_1)=t\Phi'(t)/\Phi(t)$ if $t\Phi'(t)/\Phi(t)\ge 1/\sqrt 2.$
In the latter case, $\mathcal L_f(te_1)=F(\tau_0)^{1/2}=\phi(t)/2\sqrt{1-\mathcal Q^2}.$
Hence, $T_f(te_1)^{n-1}=\mathcal L_f(te_1)^n/[\Phi'(t)\phi(t)^{n-1}]
=1/[\mathcal Q(2\sqrt{1-\mathcal Q^2})^n]$ and the required formula follows in this case, too.
\end{proof}

We can also show easily that $K_f=\max\{\mathcal Q^{n-1},\mathcal Q^{-1}\}$
and $L_f=\max\{\mathcal Q^{1-n},\mathcal Q\};$ cf. \cite{KG91}.


\subsection{Main modulus estimates}
Our main tool for studying cavitation occurrence is the following double bound for the modulus of the family
$\Gamma=\Gamma_{\mathcal A(r,R)}$ of curves which join the boundary components of the spherical ring
$\mathcal A(r,R)$ in it.
For a recent result for semirings and homeomorphisms of Sobolev class $W^{1,n-1},$ we refer to \cite[Lemma~3]{GSV25}.
Although the upper estimate has been established in \cite[Lemma~2.4]{GG09}, we provide its proof together with the newly established lower bound.
This should help the reader to avoid additional efforts in understanding and aligning notations.

\medskip
\begin{thm}\label{thm:main} 
Let $f:\mathbb B_0^n\to\mathbb R^n$ be a locally quasiconformal homeomorphism.
Then  for any nonnegative measurable functions $\rho(t),$ $t\in (r,R),$ and  $p(u),$ $u\in \mathbb S^{n-1},$
such that
\begin{equation}\label{eq:admcons}
\int\limits_r^R \rho(t)dt\,=\,1\,,\qquad\int\limits_{\mathbb S^{n-1}}p(u)^{n-1}d\sigma(u)\,=\,1\,,
\end{equation}
the following double inequality holds:
\begin{equation}\label{eq:maindoubest}
\left(\int\limits_{\,\mathcal A(r,R)}p\left(\frac{x}{|x|}\right)^n Q_f(x)\frac{dm_n(x)}{|x|^n}\right)^{1-n}\,\le\,
{\M}(f(\Gamma))\,\le\, \int\limits_{\mathcal A(r,R)}\rho^n(|x|)D_f(x)\,dm_n(x),
\end{equation}
where $\Gamma=\Gamma_{\mathcal A(r,R)}$ and
$d\sigma$ denotes the $(n-1)$-dimensional area element of the unit sphere
$\mathbb S^{n-1}$ in $\mathbb R^n.$
\end{thm}

Note that results similar to the right-hand inequality (the upper bound) are found in \cite{GG09} (for spherical rings)
and \cite{GSV25} (for spherical hemirings).
For completeness, however, we briefly give a proof of the upper bound as well.

\begin{proof}
We start with the lower bound (the left-hand side estimate).

Let $\sigma_{r,R}$ denote the family of all radial segments $\gamma_u=tu,$ $r\le t \le R,$ $u\in\mathbb S^{n-1},$
joining the boundary spheres  of the spherical annulus $\mathcal A(r,R).$  The property $f\in W^{1,n}_{\rm loc}$ guarantees absolute continuity of $f$ on almost all lines parallel to coordinate axes or equivalently (by passing to spherical coordinates), the absolute continuity of $f$ on $\gamma_u$ for almost all $u\in\mathbb S^{n-1}.$ If $\widetilde{\sigma}_{r,R}$ is the subfamily of $\sigma_{r,R}$ on which $f$ fails to be absolutely continuous, then, by Fuglede's theorem \cite{Fug57} (cf. \cite[p.~95]{Vai71}), $\M(\widetilde{\sigma}_{r,R})=0.$ Denote by
$\Sigma_{r,R}$ the family of curves $\gamma^*_u=f(\gamma_u),$ and let $\widetilde{\Sigma}_{r,R}=f(\widetilde{\sigma}_{r,R}).$ Since $\M(\widetilde{\sigma}_{r,R})=0,$ the modulus (geometric) definition of $K$-quasiconformality implies that $\M(\widetilde{\Sigma}_{r,R})=0.$ So, without loss of generality, we can assume that each curve $\gamma^*_u\in \Sigma_{r,R}$ is absolutely continuous. In addition, we will employ the well-known regularity properties of quasiconformal mappings, including differentiability a.e. and the Lusin $(N)$ and $(N^{-1})$-properties, throughout the proof; see, e.g., \cite[32.2, 33.2]{Vai71}.

First, we deduce an explicit formula (\ref{6.8}) below for the modulus ${\M}(\Sigma_{r,R})$ of
the family $\Sigma_{r,R}.$

Define a measurable non-negative function $\rho_0^*(y):\mathbb R^n\to [0,\infty]$
by setting $\rho_0^*(y)=\rho^*\circ f^{-1}(y),$ where
\begin{equation}\label{6a.4}
\rho^*(x)\,=\, \left(\frac{|f_u(tu)|}{J_f(tu)}\right)^{1/(n-1)}\frac{1}{t I(u)}
\end{equation}
at all points of regularity $x=tu (t>0, u\in\Sph)$ of $f$ in $\mathcal A(r,R)$ with
\begin{equation*}
I(u)\,=\,\int\limits_r^R
Q_f(tu) \frac{dt}{t}\,,
\end{equation*}
and set $\rho^*(x)=0$ for the rest of points $x$ in $\mathbb R^n.$

Let us show that $\rho^*_0$ is admissible for the family of curves $\Sigma_{r,R}.$
Indeed, a computation shows that
\begin{equation*}
\begin{split}
\int\limits_{\gamma_{u}^*}\rho^*_0(y)|dy|\,
&=\,\int\limits_r^R\rho^*(tu)|f_u(tu)|\,dt
=\,\int\limits_r^R\frac{|f_u(tu)|^{n/(n-1)}}{tJ_f(tu)^{1/(n-1)}I(u)}dt\, \\
&=\,\frac{1}{I(u)}\int\limits_r^R Q_f(tu)\frac{dt}{t}\,
=\,1
\end{split}
\end{equation*}
for a.e.~$u\in\Sph.$
This normalized admissible function is needed to prove the key inequality (\ref{eq:Sigma}) below.
We next show the claim that
\begin{equation}\label{6a.7}
\int\limits_{\mathbb R^n}{\rho_0(y)}^n dm_n(y)\ge \int\limits_{\mathbb R^n}\rho^*_0(y)^n dm_n(y)
\end{equation}
for any $\rho_0\in\adm(\Sigma_{r,R}).$
Since
\begin{equation*}
\begin{split}
\int\limits_{\mathbb R^n}\rho^*_0(y)^n dm_n(y)\,
&=\,\int\limits_{r<|x|<R}\rho^*(x)^nJ_f(x)\,dm_n(x)\\&
=\,\int\limits_{\mathbb S^{n-1}}\int\limits_r^R\frac{|f_u(tu)|^{n/(n-1)}t^{n-1}}%
{t^n J_f(tu)^{1/(n-1)}I(u)^n}\,dt\,d\sigma(u) \\
&=\int\limits_{\mathbb S^{n-1}}\int\limits_r^R\frac{Q_f(tu)}{I(u)^n}\frac{dt}{t}\,d\sigma(u)\,
=\,\int\limits_{\mathbb S^{n-1}}I(u)^{1-n}d\sigma(u)\,,
\end{split}
\end{equation*}
the above claim implies
$$
\M(\Sigma_{r,R})\,=\,\inf_{\rho_0\in\adm(\Sigma_{r,R})}
\int\limits_{\mathbb R^n}\rho_0(y)^ndm_n(y)\,
=\,\int\limits_{\mathbb S^{n-1}}I(u)^{1-n}d\sigma(u)\,.
$$
Hence, we have
\begin{equation}\label{6.8}
\M(\Sigma_{r,R})
=\,\int\limits_{\mathbb S^{n-1}}\left(\int_r^R Q_f(tu)\frac{dt}{t}\right)^{1-n}d\sigma(u)\,.
\end{equation}
To prove (\ref{6a.7}), we begin with the evident inequality
\begin{equation*}
\int\limits_r^R(\rho(tu)-\rho^*(tu))|f_u(tu)|dt\,
=\,\int\limits_{\gamma^*_u}(\rho_0(y)-\rho^*_0(y))|dy|\,
=\,\int\limits_{\gamma^*_u}\rho_0(y)|dy|-1\,
\geq\, 0
\end{equation*}
for a.e.~$u\in\Sph.$
If we set $\rho(x)=\rho_0(f(x)),$ and recall (\ref{6a.4}), which is equivalent to
\begin{equation*}
\rho^*(tu)^{n-1}J_f(tu)\,=\,\frac{|f_u(tu)|}{t^{n-1}I(u)^{n-1}}\,,
\end{equation*}
then we get
\begin{equation*}
\begin{split}
\int\limits_{f(\mathcal A(r,R))}(\rho_0{\rho^*_0}^{n-1}-{\rho^*_0}^n)dm_n\,&
=\,\int\limits_{\mathcal A(r,R)}(\rho-\rho^*){\rho^*}^{n-1}J_f dm_n\\
&=\,\int\limits_{\mathbb S^{n-1}}\int\limits_r^R(\rho(tu)-\rho^*(tu))\frac{|f_u(tu)|}{t^{n-1}I(u)^{n-1}}t^{n-1}dt\,d\sigma(u)\\
&=\,\int\limits_{\mathbb S^{n-1}}\left(\frac{1}{I(u)^{n-1}}\int\limits_r^R(\rho(tu)-\rho^*(tu))|f_u(tu)|\,dt\right)d\sigma(u)\,
\geq\, 0\,.
\end{split}
\end{equation*}
Thus, for an arbitrary $\rho_0\in\adm(\Sigma_{r,R}),$ we have
\begin{equation*}
\int\limits_{\mathbb R^n}{\rho^*_0}^n dm_n\,
\le\,\int\limits_{\mathbb R^n}\rho_0{\rho^*_0}^{n-1}dm_n\,.
\end{equation*}
On the other hand, by the H\"older inequality,
\begin{equation*}
\int\limits_{\mathbb R^n}\rho_0{\rho^*_0}^{n-1}dm_n\,
\leq\, \left(\,\int\limits_{\mathbb R^n}\rho_0^{n}dm_n\right)^{1/n}
\left(\int\limits_{\,\mathbb R^n}{\rho^*_0}^{n}dm_n\right)^{(n-1)/n}\,.
\end{equation*}
Thus,
\begin{equation*}
\int\limits_{\mathbb R^n}{\rho^*_0}^n dm_n\,
\leq\, \left(\,\int\limits_{\mathbb R^n}\rho_0^{n}dm_n\right)^{1/n}
\left(\,\int\limits_{\mathbb R^n}{\rho^*_0}^{n}dm_n\right)^{(n-1)/n}\,,
\end{equation*}
which implies (\ref{6a.7}).

Next, by the Comparison Principle, we obtain
\begin{equation*}
{\M}(\Sigma_{r,R})\,\leq\, {\M}(f(\Gamma))\,.
\end{equation*}
Hence,
\begin{equation}\label{eq:Sigma}
{\M}(f(\Gamma))\,\geq\, \int\limits_{\mathbb S^{n-1}}\frac{d\sigma(u)}{I(u)^{n-1}}\,
=\,\int\limits_{\mathbb S^{n-1}}\left(\int_r^R Q_f(tu)\frac{dt}{t}\right)^{1-n}d\sigma(u),
\end{equation}
which provides the required lower estimate.

To obtain the desired lower bound in (\ref{eq:maindoubest})
we will proceed in a standard manner.
Let $p:\mathbb S^{n-1}\to {\mathbb R}$ be a nonnegative measurable function
satisfying the second equality in (\ref{eq:admcons}).
We rewrite that equality as
\begin{equation*}
1\,=\,\int\limits_{\mathbb S^{n-1}}\frac{1}{I(u)^{(n-1)/n}}\, p(u)^{n-1}I(u)^{(n-1)/n}d\sigma(u)\,,
\end{equation*}
and apply the H\"older inequality with exponents $n$ and $n/(n-1)$ to obtain
\begin{equation*}
1\,\leq\, \left(\int\limits_{\,\,\mathbb S^{n-1}}\frac{d\sigma(u)}{I(u)^{n-1}}\right)^{1/n}
\left(\int\limits_{\,\,\mathbb S^{n-1}}p(u)^n I(u)\,d\sigma(u)\right)^{1-1/n}.
\end{equation*}
Thus, we arrive at the inequality
\begin{equation*}
\int\limits_{\,\mathbb S^{n-1}}\frac{d\sigma(u)}{I(u)^{n-1}}\,
\geq\, \left(\,\int\limits_{\,\mathbb S^{n-1}}p(u)^n I(u)d\sigma(u)\right)^{1-n}.
\end{equation*}
Substituting the expression in (\ref{6.8}) into the above inequality yields
\begin{equation*}
\begin{split}
{\M}(f(\Gamma))\,&\geq\, \left(\,\int\limits_{\,\mathbb S^{n-1}}p(u)^nI(u)\,d\sigma(u)\right)^{1-n}
\\&=\,\left(\,\int\limits_{\,\mathcal A(r,R)}p\left(\frac{x}{|x|}\right)^n Q_f(x)\frac{dm_n(x)}{|x|^n}\right)^{1-n}\,.
\end{split}
\end{equation*}
This completes the proof of the lower bound in (\ref{eq:maindoubest}).


\medskip

We now prove the upper bound. Fix a spherical ring $\mathcal A=\mathcal A(r,R),$ $0<r<R\le 1.$ Let $\mathcal N$ denote the set of all points of $\mathcal A$ where $f$ is either non-differentiable or differentiable but $J_f(x)=0.$ The $n$-dimensional Lebesgue measure of $\mathcal N$ vanishes, which follows from the a.e.~differentiability of quasiconformal mappings and the validity of the Lusin $(N^{-1})$-property. Let $\gamma\in\Gamma_{\mathcal A}$ and $\gamma^*=f(\gamma)$ be its image under $f.$ Note that due to Fuglede's theorem \cite{Fug57}, we can restrict ourselves considering only rectifiable curves $\gamma$ and $\gamma^*$ since the conformal modulus of subfamilies consisting of non-rectifiable curves vanishes.

Let $\rho$ be a nonnegative measurable function on $(r,R)$ satisfying the first equality in (\ref{eq:admcons}).
Then we define $\rho_0$ by
\begin{equation*}
\rho_0(y)\,=\,\frac{\rho(|x|)}{\ell_f(x,0)}
\end{equation*}
for $y\in f(\mathcal A\setminus\mathcal N),$
where $x=f^{-1}(y),$ and $\rho_0(y)=\infty$ for $y\in f(\mathcal N),$ and  $\rho_0(y)=0$ otherwise.
We will show that $\rho_0\in \adm(f(\Gamma_\A)).$

Take $\gamma\in\Gamma_\A$ such that $\gamma$ and $\gamma^*$ are rectifiable.
We parametrize $\gamma$ by the arclength parameter $s.$
If $f$ is regular at $x=\gamma(s),$ then $df(\gamma(s))/ds=f'(x)h=f_h(x),$ where $h=\gamma'(s)\in\Sph$ for a.e.~$s.$
We further let  $t=|x|=|\gamma(s)|$ and $u=x/|x|.$
Then, $dt/ds\le |dt/ds|=|h\cdot u|.$
Also, by the definition of $\ell_f(x,0)=\ell_f(x),$ we have $\ell_f(x)\le |f_h(x)|/|h\cdot u|,$
equivalently, $|h\cdot u|\le |f_h(x)|/\ell_f(x).$
Hence,
$$
\int\limits_{\gamma^*} \rho_0(y)|dy|\,
=\,\int\limits \frac{\rho(|x|)}{\ell_f(x)}|f_h(x)|\,|dx|
\ge\,\int\limits_{\gamma}\rho(t)|h\cdot u|\,ds
\ge\,\int\limits_r^R\rho(t)dt\,
=1,
$$
thus it has been confirmed that $\rho_0\in\adm(f(\Gamma_\A)).$
To complete the proof of the upper bound of (\ref{eq:maindoubest}),
we use the change of variables formula together with the absolute continuity in measure (the Lusin $(N)$-property),
\begin{equation*}
\M(f(\Gamma))\,\le\,\int\limits_{f(\mathcal A)}\rho_0(y)^n dm_n(y)\,
=\,\int\limits_{\mathcal A}\rho(|x|)^n D_f(x)\,dm_n(x)\,,
\end{equation*}
which completes the proof.
\end{proof}

\medskip
The sharp lower bound for the modulus of $f(\Gamma)$ in (\ref{eq:maindoubest}) has been given in (\ref{eq:Sigma}) (in terms of the normal directional dilatation $Q_f$). Since the upper estimate in (\ref{eq:maindoubest}) holds for any admissible metric $\rho(|x|)$ satisfying the first equality in (\ref{eq:admcons}), we derive the extremal upper bound for $\M(f(\Gamma))$ now in terms of the angular directional dilatation $D_f.$

\begin{thm}\label{thm:thm4} 
Let $f:\pcub\to \mathbb R^n$ be a locally quasiconformal embedding.
Then for any $0<r<R\le 1,$
\begin{equation}\label{eq:thm4}
\begin{split}
\int\limits_{\mathbb S^{n-1}}\left(\int_r^R Q_f(tu)\frac{dt}{t}\right)^{1-n}d\sigma(u) \,
&\le\,\M(f(\Gamma_{\A(r,R)}))\,\\
&\le\,\left(\int\limits_r^R\left(\int_{\mathbb S^{n-1}} D_f(tu)t^{n-1}d\sigma(u)\right)^{1/(1-n)}
dt\right)^{1-n}\,.
\end{split}
\end{equation}
\end{thm}

Employing the classical inner and outer dilatations and their relations (\ref{eq:Qf}) with the directional dilatations we obtain a weaker double bound than (\ref{eq:thm4}).

\begin{cor}\label{cor:cor1}
Let $f:\pcub\to \mathbb R^n$ be a locally quasiconformal  embedding.
Then for any $0<r<R\le 1,$
\begin{equation*}
\begin{split}
\int\limits_{\mathbb S^{n-1}}\left(\int_r^R K_f(tu)\frac{dt}{t}\right)^{1-n}d\sigma(u) \,
&\le\,\M(f(\Gamma_{\A(r,R)}))\,\\
&\le\,\left(\int\limits_r^R\left(\int_{\mathbb S^{n-1}} L_f(tu)t^{n-1}d\sigma(u)\right)^{1/(1-n)}
dt\right)^{1-n}\,.
\end{split}
\end{equation*}
\end{cor}

\begin{proof}[Proof of Theorem \ref{thm:thm4}]

The first inequality was already established in (\ref{eq:Sigma}).
We show the second inequality.
Consider
\begin{equation*}
\widehat\rho(t)\,=\,I^{-1}\varphi(t)^{1/(1-n)}\,,\qquad
\varphi(t)\,=\,\int\limits_{\mathbb S^{n-1}} D_f(tu)t^{n-1}\,d\sigma(u)\,,\qquad
I\,=\,\int\limits_r^R \varphi(t)^{1/(1-n)} dt\,.
\end{equation*}
Evidently, $\widehat\rho(t)$ satisfies the condition $\int_r^R\widehat\rho(t)dt=1.$
We shall show that $\widehat\rho(t)$ attains the infimum on the right-hand side of (\ref{eq:maindoubest})
over all admissible $\rho(t).$

Using the H\"older inequality for an arbitrary metric $\rho(t)$, we have
\begin{equation*}
1\,=\,\int\limits_r^R \rho(t)\,dt\,\le\left(\,\int\limits_r^R\rho(t)^n\varphi(t)\,dt\right)^{1/n} \left(\,\int\limits_r^R\varphi(t)^{1/(1-n)}dt\right)^{(n-1)/n}\,,
\end{equation*}
and equality holds whenever $\rho(t)^n\varphi(t)$ is proportional to $1/\varphi(t)^{1/(n-1)}.$ Clearly, $\widehat\rho(t)$ meets this requirement, since $\widehat\rho(t)\varphi(t)^{1/(n-1)}=I^{-1}.$
Therefore,
\begin{equation*}
\int\limits_r^R \widehat\rho(t)^n\varphi(t)\,dt\,=\,\inf\limits_{\rho}\int\limits_r^R \rho(t)^n\varphi(t)\,dt\,=\,\left(\,\int\limits_r^R\varphi(t)^{1/(1-n)}dt\right)^{1-n}\,,
\end{equation*}
where the infimum is taken over all $\rho$ satisfying the equality in (\ref{eq:admcons}). Finally, by Fubini's theorem for any admissible metric $\rho,$ one gets
\begin{equation*}
\begin{split}
\int\limits_{\mathcal A(r,R)} \rho(|x|)^n D_f(x)\,dm_n(x)\,&=\,\int\limits_r^R \rho(t)^n\left(\,\,\int\limits_{\mathbb S^{n-1}} D_f(tu)t^{n-1}d\sigma(u)\right)\,dt\,\\&\ge\,\left(\,\int\limits_r^R\varphi(t)^{1/(1-n)}dt\right)^{1-n}\,=\,I^{1-n}\,.
\end{split}
\end{equation*}
This completes the proof.
\end{proof}


\section{Cavitation and main results}\label{sec3}

In this section, we establish conditions for cavitation and non-cavitation at the origin.
The following theorem provides a necessary and sufficient condition for non-cavitation.

\begin{thm}\label{thm:necsufcond} 
Let $f:\pcub\to \mathbb R^n$ be a locally quasiconformal embedding.
Then $f$ does not admit cavitation at the origin if and only if
\begin{equation*}
\lim_{r\to 0^+} \mo\, f(\A(r,1))=+\infty\,.
\end{equation*}
\end{thm}

For convenience, we assume that the ring $f(\B^n_0)$ separates $0$ from $\infty.$
When $f(\A(r,1))$ is expressed as $\ring(C_0(r), C_1),$ set
$$
R_0(r)=\max_{x\in C_0(r)}|x| \aand R_1=\min_{x\in C_1}|x|\,.
$$
Then $R_0(r)\le \diam C_0(r) \le 2R_0(r).$
Thus, $f$ does not admit cavitation precisely when $R_0(r)\to 0$ as $r\to0.$
The theorem now follows from the next lemma as a special case when $r\to 0.$

\begin{lem}
Let $f:\pcub\to \mathbb R^n$ be a locally quasiconformal embedding
such that $f(\B^n_0)$ separates $0$ from $\infty.$
If $\mathcal M(r)=\mo f(\A(r,1))$ is greater than $A_n,$
$$
R_1 e^{-\mathcal M(r)}\,\le\, R_0(r)\,\le\, R_1 e^{A_n-\mathcal M(r)}\,,
$$
where $A_n$ is the constant given in \eqref{eq:An}.
\end{lem}

\begin{proof}
Fix $r\in(0,1)$ and set $\ring_r=f(\A(r,1)).$
By the condition $\mathcal M(r)>A_n,$
Theorem~\ref{thm:Teich} implies that there is a subring $\A(\hat r, \hat R)$ of $\ring_r$
for some $0<\hat r<\hat R$ with $\log(\hat R/\hat r)\ge \mathcal M(r)-A_n.$
In particular, $R_0(r)\le\hat r<\hat R\le R_1.$
Hence,
$$
\frac{R_1}{R_0(r)}\,\ge\, \frac{\hat R}{\hat r}\,\ge\, e^{\mathcal M(r)-A_n}\,,
$$
and the right-hand inequality follows.
On the other hand, since $\A(R_0(r),R_1)$ is a subring of $\ring_r,$
we have $\mod\A(R_0(r),R_1)=\log(R_1/R_0(r))\le\mo \ring_r=\mathcal M(r),$
and, therefore, $R_1/R_0(r)\le e^{\mathcal M(r)},$ which is equivalent to
the left-hand inequality in the assertion.
\end{proof}

Now, combining the modular criterion of Theorem~\ref{thm:necsufcond} and the left hand side estimate in
Theorem~\ref{thm:thm4}, we obtain the following sufficient condition for
cavitation.

\begin{thm}\label{thm:suf1} 
Let $f:\pcub\to\mathbb R^n$ be a locally quasiconformal embedding with
\begin{equation}\label{P}
I_Q\,=\,I_Q(f)\,
:=\,\int\limits_{\mathbb S^{n-1}}\left(\int_0^1 Q_f(tu)\frac{dt}{t}\right)^{1-n}d\sigma(u)\,>\,0\,.
\end{equation}
Then $f$ admits cavitation at the origin.
\end{thm}

Note that condition (\ref{P}) is equivalent to the convergence of the integral
$$\mathfrak{I}_f(u)\ :=\ \int_0^1 Q_f(tu)\frac{dt}{t}\ <\ +\infty$$ on a set of directions $u\in\mathbb{S}^{n-1}$ of positive
spherical measure.

\medskip

By Theorem \ref{thm:suf1} we have also the necessary condition for the absence of cavitation:

\begin{cor}\label{cor:nec1} 
Let $f:\pcub\to\mathbb R^n$ be a locally quasiconformal embedding.
If $f$ does not admit cavitation at the origin, then $I_Q(f)=0$.
\end{cor}

Note also that the condition $I_Q(f)=0$ is equivalent to the divergence of the integral
$\mathfrak{I}_f(u).$

\begin{rem}\label{G}
Theorem~\ref{thm:suf1} formally gives the following sufficient condition
for cavitation:
\begin{equation*}
I_K\,=\,I_K(f)\,
:=\,\int\limits_{\mathbb S^{n-1}}\left(\int_0^1 K_f(tu)^{1/(n-1)}\frac{dt}{t}\right)^{1-n}d\sigma(u)\,>\,0\,,
\end{equation*}
because by \eqref{eq:Qf}, $Q_f(x)\le K_f(x)^{1/(n-1)}$ a.e. and, consequently, $0\le I_K(f)\le I_Q(f)$. Similarly, Corollary~\ref{cor:nec1} formally provides a necessary condition $I_K(f)=0$ for non-cavitation.

However, since $K_f(x)\ge 1$ a.e., spherical coordinates yield
\begin{equation*}
\int_0^1 K_f(tu)^{1/(n-1)}\,\frac{dt}{t}
\ge \int_0^1\frac{dt}{t}=+\infty
\qquad\text{for $\sigma$-almost every }u\in\mathbb S^{n-1}.
\end{equation*}
Thus, the above formal sufficient condition can never be satisfied,
whereas the given formal necessary condition holds automatically, independently of
whether cavitation occurs. These conditions therefore provide no
information about cavitation.
This illustrates the role of the normal dilatation $Q_f$, whose
replacement by the classical upper bound $K_f^{1/(n-1)}$ loses
the information needed for this cavitation criterion.
\end{rem}

Similarly, combining the modular criterion of Theorem~\ref{thm:necsufcond} and the right-hand side estimate in
Theorem~\ref{thm:thm4}, we obtain the following sufficient condition for
non-cavitation.

\begin{thm}\label{thm:suf3} 
Let $f:\pcub\to\mathbb R^n$ be a locally quasiconformal embedding with
\begin{equation*}
I_D\,=\,I_D(f)\,
:=\,\int\limits_0^1\left(\int_{\mathbb S^{n-1}} D_f(tu)t^{n-1}d\sigma(u)\right)^{1/(1-n)}dt\,=\,+\infty\,.
\end{equation*}
Then $f$ does not admit cavitation at the origin.
\end{thm}

\begin{cor}\label{cor:nec3} 
Let $f:\pcub\to\mathbb R^n$ be a locally quasiconformal embedding.
If $f$ admits cavitation at the origin, then $I_D(f)<+\infty.$
\end{cor}

Finally, in view of \eqref{eq:Qf}, by replacing the angular directional dilatation with the inner one, we obtain the following result:

\begin{thm}\label{thm:suf4} 
Let $f:\pcub\to\mathbb R^n$ be a locally quasiconformal embedding with
\begin{equation*}
I_L\,=\,I_L(f)\,
:=\,\int\limits_0^1\left(\int_{\mathbb S^{n-1}} L_f(tu)t^{n-1}d\sigma(u)\right)^{1/(1-n)}dt\,=\,+\infty\,.
\end{equation*}
Then $f$ does not admit cavitation at the origin.
\end{thm}

\begin{cor}\label{cor:nec4} 
Let $f:\pcub\to\mathbb R^n$ be a locally quasiconformal embedding.
If $f$ admits cavitation at the origin, then $I_L(f)<+\infty.$
\end{cor}


\section{Illustrating examples}\label{sec4}

We now present several illustrative examples, including the ones mentioned in the Introduction, to examine the results established in the previous section.

\medskip
Applying Lemma~\ref{lem:rs} to $f_1=\frac{1+|x|^\alpha}{2|x|}x$ with $0<\alpha<1,$ we have
$\mathcal Q=\alpha/(1+t^{-\alpha})<1$ and
\begin{equation*}
L_{f_1}(tu)\,=\,\left(\frac{1+t^{-\alpha}}{\alpha}\right)^{n-1}\,,\quad K_{f_1}(tu)\,=\,\frac{1+t^{-\alpha}}{\alpha}\,, \quad J_{f_1}(tu)\,=\,\frac{\alpha(1+t^\alpha)^{n-1}}{2^nt^{n-\alpha}}\,,
\end{equation*}
and
\begin{equation*}
D_{f_1}(x)\,=\,\left(\frac{1+t^{-\alpha}}{\alpha}\right)^{n-1}\,,\quad Q_{f_1}(x)\,=\,\frac{\alpha}{1+t^{-\alpha}}\,.
\end{equation*}

Thus, in view of Theorem~\ref{thm:suf1}, a calculation of $I_Q$ yields:
\begin{equation*}
\begin{split}
I_Q&=\int\limits_{\mathbb S^{n-1}}\left(\int_0^1 Q_f(tu)\frac{dt}{t}\right)^{1-n}d\sigma(u)\\
&=\int\limits_{\mathbb S^{n-1}}\left(\int_0^1 \frac{\alpha}{1+t^{-\alpha}}\frac{dt}{t}\right)^{1-n}d\sigma(u)\\
&=\omega_{n-1}\left(\int\limits_0^1\frac{\alpha t^{\alpha-1}dt}{t^\alpha+1}\right)^{1-n}
=\omega_{n-1}(\log 2)^{1-n} > 0,
\end{split}
\end{equation*}
and cavitation occurs at the origin.

Importantly, the sufficient condition for cavitation in Remark \ref{G}, which is based on the classical dilatation $K_f,$ is not satisfied, as the integral $I_K$ evaluates to zero:
\begin{equation*}
I_K\,=\,\int\limits_{\mathbb S^{n-1}}\left(\int_0^1 \left(\frac{1+t^{-\alpha}}{\alpha}\right)^{1/(n-1)}\frac{dt}{t}\right)^{1-n}d\sigma(u)=0,
\end{equation*}
because the integral in parentheses diverges.

The result $I_D < \infty$ and $I_L < \infty$ confirms consistency with the necessary conditions for cavitation (Corollaries~\ref{cor:nec3} and \ref{cor:nec4}), since both $I_D$ and $I_L$ are finite:
\begin{equation*}
\begin{split}
I_D\,=\,I_L\,&=\,\int\limits_0^1\left(\int\limits_{\, \mathbb S^{n-1}}D_{f_1}(tu)t^{n-1}d\sigma(u)\right)^{1/(1-n)}dt
\\&=\,\int\limits_0^1\left(\int\limits_{\, \mathbb S^{n-1}}\frac{(1+t^{-\alpha})^{n-1}}{\alpha^{n-1}}t^{n-1}d\sigma(u)\right)^{1/(1-n)}dt
\\&=\,\alpha\omega_{n-1}^{1/(1-n)}\int\limits_0^1\frac{dt}{t^{1-\alpha}(1+t^\alpha)}\,<\,+\infty\,.
\end{split}
\end{equation*}

\medskip
The stretching $f_2=x e^{1-1/|x|}$ has no cavitation at the origin. For this mapping, Lemma~\ref{lem:rs} yields
\begin{equation*}
L_{f_2}(x)\,=\,1+\frac{1}{t}\,,\quad
J_{f_2}(x)\,=\,e^{n(1-1/t)}\left(1+\frac{1}{t}\right)\,,
\end{equation*}
and
\begin{equation*}
D_{f_2}(x)\,=\,\left(1+\frac{1}{t}\right)^{1-n}\,,\qquad Q_{f_2}(x)\,=\,1+\frac{1}{t}\,.
\end{equation*}

Consequently, for $f_2,$ the calculation for $I_D$ is:
\begin{equation*}
\begin{split}
I_D\,&=\,\int\limits_0^1\left(\int\limits_{\,\mathbb S^{n-1}}D_f(tu)t^{n-1}d\sigma(u)\right)^{1/(1-n)}dt
\\&=\,\int\limits_0^1\left(\int\limits_{\,\mathbb S^{n-1}}\left(1+\frac{1}{t}\right)^{1-n}t^{n-1}d\sigma(u)\right)^{1/(1-n)}dt
\\&=\,\omega_{n-1}^{1/(1-n)}\int\limits_0^1\frac{(1+t)}{t^2}dt\,=\,+\infty\,,
\end{split}
\end{equation*}
and, by Theorem~\ref{thm:suf3}, $f_2$ has no cavitation at the origin.

In contrast, the sufficient condition involving the inner dilatation (Theorem~\ref{thm:suf4}) is inconclusive, as the integral $I_L$ is finite:
\begin{equation*}
\begin{split}
I_L\,&=\,\int\limits_0^1\left(\int\limits_{\, \mathbb S^{n-1}}L_f(tu)t^{n-1}d\sigma(u)\right)^{1/(1-n)}dt
\\&=\,\int\limits_0^1\left(\int\limits_{\, \mathbb S^{n-1}}\left(1+\frac{1}{t}\right)t^{n-1}d\sigma(u)\right)^{1/(1-n)}dt
\\&=\,\omega_{n-1}^{1/(1-n)}\int\limits_0^1(1+t)^{1/(1-n)}t^{-(n-2)/(n-1)}dt\,<\,+\infty\,.
\end{split}
\end{equation*}

The necessary condition for non-cavitation in Corollary~\ref{cor:nec1} is also confirmed:
\begin{equation*}
I_Q=\int\limits_{\mathbb S^{n-1}}\left(\int_0^1 Q_f(tu)\frac{dt}{t}\right)^{1-n}d\sigma(u)
=\int\limits_{\mathbb S^{n-1}}\left(\int_0^1\left(1+\frac{1}{t}\right)\frac{dt}{t}\right)^{1-n}d\sigma(u)=0,
\end{equation*}
because, as we see, the integral $\mathfrak{I}_f$ diverges.

\medskip
Finally, we analyze the rotation $f_3$ by calculating the integrals $I_Q,$ $I_D$ and $I_L,$ and drawing conclusions regarding the sufficient and necessary conditions presented in Section~\ref{sec3}.
Recall that for the Lipschitz mapping $f_3,$
\begin{equation*}
L_{f_3}(x)\,=\,\left(\sqrt{2}+1\right)^n\,,\qquad
J_{f_3}(x)\,=\,1\,,
\end{equation*}
see, e.g., \cite{GG09}.
The directional dilatations are computed as follows:
\begin{equation*}
D_{f_3}(x)\,=\,1\,,\qquad Q_{f_3}(x)\,=\,\left(1+\frac{4|z|^2}{|x|^2}\right)^{n/2(n-1)}\,.
\end{equation*}

Let us give computations of  the corresponding three integrals:
\begin{equation*}
\begin{split}
I_Q\,&=\,\int\limits_{\mathbb S^{n-1}}\left(\int_0^1 Q_f(tu)\frac{dt}{t}\right)^{1-n}d\sigma(u)\,\\&=\,
\int\limits_{\mathbb S^{n-1}}\left(\int_0^1 \left(1+4(u_1^2+u_2^2)\right)^{n/2(n-1)}\!\frac{dt}{t}\right)^{1-n}d\sigma(u)\,=\,0\,,
\end{split}
\end{equation*}
because the integral in parentheses diverges, and
\begin{equation*}
I_D\,=\,\int\limits_0^1\left(\,\int\limits_{\, \mathbb S^{n-1}}D_f(tu)t^{n-1}d\sigma(u)\right)^{1/(1-n)}\!dt\,=\,
\omega_{n-1}^{1/(1-n)}\int_0^1 \frac{dt}{t}\,=\,+\infty\,,
\end{equation*}
\begin{equation*}
I_L\,=\,\int\limits_0^1\left(\int\limits_{\, \mathbb S^{n-1}}L_f(tu)t^{n-1}d\sigma(u)\right)^{1/(1-n)}\!dt\,=\,
\omega_{n-1}^{1/(1-n)}\left(\sqrt{2}-1\right)^{n/(n-1)}\int_0^1 \frac{dt}{t}\,=\,+\infty\,.
\end{equation*}
Thus, these three equalities successfully predict the absence of cavitation for mapping $f_3$, too.


\section{Applications of modulus estimates}\label{sec5}

In this section, we also show that the known estimate (\ref{t3.3}) below can be easily derived from Theorem~\ref{thm:main}.

\medskip

Namely, in the paper \cite{BGMV03} (see Corollary~2.8), the authors have proven the inequality
\begin{equation}\label{t3.3}
\log\frac{m_f(R)}{M_f(r)}-\log\frac{R}{r}\,\le\, \frac{1}{\omega_{n-1}}\int\limits_{\mathcal A(r,R)}\, \frac{L_f(x)-1}{|x|^n}\,dm_n(x)\,,
\end{equation}
where
\begin{equation*}
M_f(t)\,=\,\max_{|x|=t}\,|f(x)|\,,\qquad m_f(t)\,=\,\min_{|x|=t}\,|f(x)|\,.
\end{equation*}

The proof of (\ref{t3.3}) in \cite{BGMV03} relies on a fundamental result of Gehring \cite{Geh62} (see, also, \cite[p.~27]{Vai71},
\cite[p.~58]{Car74}, \cite[p.~108]{Resh89}),
the proof of which, in turn, involves a combination
of some space modulus technique with Hardy--Littlewood--Polya's symmetrization
principle.

\medskip
Now, let us set $p(u)=1/\omega_{n-1}^{1/(n-1)}$, where $\omega_{n-1}$ is the
$(n-1)$-dimensional surface area of the unit sphere $S^{n-1}$ in ${\mathbb R}^n.$
We see that the function $p(u)$ satisfies the second equality in (\ref{eq:admcons}), and then the lower bound in (\ref{eq:maindoubest}) from Theorem~\ref{thm:main} yields
\begin{equation*}
{\M}(f(\Gamma))\ge \omega_{n-1}^{n}\left(\int_{\mathcal A(r,R)}Q_f(x)\frac{dm_n(x)}{|x|^n}\right)^{1-n}\ ,
\end{equation*}
where $\Gamma$ is the family of curves joining the boundary components in $\mathcal A(r,R).$

We arrive from here at the inequality
\begin{equation}\label{t3.3+}
\mo f(\mathcal A(r,R))\,\le\, \frac{1}{\omega_{n-1}}\int\limits_{\mathcal A(r,R)}\, \frac{L_f(x)}{|x|^n}\,dm_n(x)
\end{equation}
because by \eqref{eq:Qf} $Q_f(x)\le L_f(x)$ a.e., where
\begin{equation*}
\mo f(\mathcal A(r,R))\,=\,\left[\frac{\omega_{n-1}}{\M(f(\Gamma))}\right]^{1/(n-1)}\ .
\end{equation*}
By the way, if $L_f(x)\le K$ a.e., we come from (\ref{t3.3+}) to the fundamental inequality
\begin{equation*}
\mo f(\mathcal A(r,R))\,\le\, K\mo \mathcal A(r,R)\,.
\end{equation*}

If $m_f(R)\le M_f(r),$ then the inequality (\ref{t3.3}) is trivial. Finally, if $m_f(R)> M_f(r),$ then
$f(\mathcal A(r,R))\supset \mathcal A(M_f(r),m_f(R)),$ and hence (\ref{t3.3+}) implies that
\begin{equation*}
\log\frac{m_f(R)}{M_f(r)}\,\le\, \mo f(\mathcal A(r,R))\,\le\, \frac{1}{\omega_{n-1}}\int\limits_{\mathcal A(r,R)}\, \frac{L_f(x)}{|x|^n}\,dm_n(x)\ ,
\end{equation*}
i.e., the required inequality (\ref{t3.3}) follows.


\bigskip
\noindent
{\bf Declarations.}

Conflict of Interest: None.

Ethical Approval: Not applicable.

Data Availability: No datasets were generated or analyzed during this study.

\bigskip
\noindent
{\bf Acknowledgements.}
The second and third authors were partially supported by the grant from the Simons Foundation (SFI-PDUkraine-00017674, V.~Gutlyanski\u\i\ and V.~Ryazanov). They were also supported by the grant ``Mathematical modelling of complex systems and processes related to security'' of the National Academy of Sciences of Ukraine under the budget programme ``Support for the development of priority areas of scientific research'' for 2025–2026 (p/n 0125U000299).

\bigskip

The authors also thank the referees for their careful reading of the manuscript and for their helpful comments and suggestions.


\end{document}